\documentclass{amsart}
\usepackage{amsmath}
\usepackage{latexsym}
\usepackage{amssymb}
\usepackage{epsfig}
\usepackage{amsthm}
\usepackage{latexsym}
\usepackage{longtable}
\usepackage{epsfig}
\usepackage{amsmath}
\usepackage{hhline}
\usepackage{color}
\usepackage{tikz-cd}
\usepackage{amscd}
\newtheorem{theorem}{Theorem}[section]

\newtheorem{lm}[theorem]{Lemma}
\newtheorem{tr}[theorem]{Theorem}
\newtheorem{cor}[theorem]{Corollary}

\newtheorem{rem}[theorem]{Remark}
\newtheorem{pr}[theorem]{Proposition}

\DeclareMathOperator{\GL}{GL}

\begin{document}
\title[Presentation of rational Schur algebras]{Presentation of rational Schur algebras}
\author{Franti\v sek Marko}
\address{The Pennsylvania State University, 76 University Drive, Hazleton, 18202 PA, USA}
\email{fxm13@psu.edu}
\begin{abstract}{We present rational Schur algebra $S(n,r,s)$ over an arbitrary ground field $K$ as a quotient of the distribution algebra $Dist(G)$ of the general linear group $G=\GL(n)$ by an ideal $I(n,r,s)$ and provide an explicit description of the generators of $I(n,r,s)$. 
Over fields $K$ of characteristic zero, this corrects and completes a presentation of $S(n,r,s)$ in terms of generators and relations originally considered by Dipper and Doty. The explicit presentation over ground fields of positive characteristics appears here for the first time.}
\end{abstract}
\maketitle

\section*{Introduction}
Denote by $G=\GL(n)$ the general linear group over a ground field $K$, 
$\mathfrak{gl}_n$ the corresponding general linear Lie algebra, $T$ a maximal torus of $G$, and 
$E=K^n$ the natural $G$-module of dimension $n$.

First, assume that the characteristic of the ground field $K$ is zero. 
Then the Schur algebra $S(n,d)$ is the image of the natural morphism 
\[\rho_d: U(\mathfrak{gl}_n)\to End_K(E^{\otimes d})\]
from the universal enveloping algebra $U(\mathfrak{gl}_n)$ of the general Lie algebra $\mathfrak{gl}_n$ to endomorphisms of the $d$th tensor product of $E$.

Denote by $\{\alpha_1, \ldots, \alpha_{n-1}\}$ a set of simple positive roots of $\mathfrak{gl}_n$, by $\Lambda$ the weight lattice of $\mathfrak{gl}_n$, by $\epsilon_1, \ldots, \epsilon_n$ the standard basis of $\Lambda$, and by $(.|.)$ the Killing form of $\mathfrak{gl}_n$. 
The algebra $U(\mathfrak{gl}_n)$ is given by generators $e_i, f_i$ for $i=1, \ldots, n-1$, and $H_i$ for $1\leq i\leq n$
subject to the following relations $(a)$ through $(e)$: 

$(a) H_iH_j=H_jH_i$.

$(b) e_if_j-f_je_i=\delta_{ij} (H_j-H_{j+1})$.

$(c) H_ie_j-e_jH_i=(\epsilon_i,\alpha_j)e_j, \quad H_if_j-f_jH_i=-(\epsilon_i,\alpha_j)f_j$.

$(d) e_i^2e_j-2e_ie_je_i+e_je_i^2=0$ if $|i-j|=1$; and $e_ie_j-e_je_i=0$ otherwise.

$(e) f_i^2f_j-2f_if_jf_i+f_jf_i^2=0$ if $|i-j|=1$; and $f_if_j-f_jf_i=0$ otherwise.

In the paper \cite{dg}, Doty and Giaquinto showed that the kernel of $\rho_d$ is generated by the ideal of 
$U(\mathfrak{gl}_n)$ spanned by elements
\begin{itemize}\item $H_1+\ldots+ H_n-d$ and \item $H_i(H_i-1)\ldots (H_i-d+1)(H_i-d)$ for $1\leq i\leq n$.
\end{itemize} 
As a consequence, they obtained a presentation of $S(n,d)$ by generators $e_i, f_i$ for $i=1, \ldots, n-1$, and $H_i$ for $1\leq i\leq n$ and above relations $(a)-(e)$ together with 

$(f) H_1+\ldots + H_n=d$.

$(g) H_i(H_i-1)\ldots (H_i-d+1)(H_i-d)=0$ for $1\leq i\leq n$.

The defining relations for generalized $q$-Schur algebras were determined in \cite{doty} while a different set of generators and relations was given in \cite{dgs}.

Now assume that the ground field $K$ has an arbitrary characteristic and denote by $Dist(G)$ the distribution algebra of $G$. Then there is a natural morphism 
\[\rho_d: Dist(G)\to End_K(E^{\otimes d})\]
such that its image coincides with the Schur algebra $S(n,d)$.

We recall the definition of the generalized Schur algebra given by Donkin in \cite{d}. We call a $Dist(G)$-module admissible if it is a direct sum of its weight spaces.
Let $\mu$ be a finite set of dominant weights of $G$ and $I(\mu)$ be the ideal of $Dist(G)$ consisting of all elements that annihilate all finite-dimensional admissible $Dist(G)$-modules
that have composition factors solely from the set of irreducible modules $L(\lambda)$ of the highest weight $\lambda$ for $\lambda\in \mu$.
The generalized Schur algebra $S(\mu)$ corresponding to the set $\mu$ is the factoralgebra of $Dist(G)$ by the ideal $I(\mu)$.

Dipper and Doty introduced rational Schur algebras $S(n,r,s)$ in \cite{dd}. By \cite{dd,d}, the rational Schur algebra $S(n,r,s)=S(\pi)$ is a generalized Schur algebra corresponding to a saturated set $\pi$ of dominant weights of $G$. 
The set $\pi=\Lambda^+(n,r,s)$ consists of dominant weights in the set of weights $\Lambda(n,r,s)$ appearing in the weight space decomposition of the mixed tensor product space 
$E^{r,s}_{K}=E_K^{\otimes r}\otimes {E^*_{K}}^{\otimes s}$, where $E$ is the natural $G$-module.

In their attempt to describe $S(n,r,s)$, Dipper and Doty defined in \cite[Section 7.3]{dd} an algebra $S'(n,r,s)$ by generators and relations. Donkin in \cite[Subsection 4.2]{d} showed that $S'(n,r,s)$ is not isomorphic to $S(n,r,s)$ in general but proved that $S'(n,r,s)=S(\pi')$ is a generalized Schur algebra for a saturated set $\pi'$. Here $\pi'$ consists of dominant weights 
of the set of weights $\lambda=(\lambda_1, \ldots, \lambda_n)$ of the Schur algebra $S(n,r+(n-1)s)$  satisfying 
$0\leq \lambda_1, \ldots, \lambda_n\leq r+s$.

A bijective correspondence exists between the weights in $\pi$ and a subset $\pi''\subset \pi'$ given in Lemma \ref{lm1.5}. The surjection 
\[S(n, r+(n-1)s)\twoheadrightarrow S(\pi') \twoheadrightarrow S(\pi'') \simeq S(\pi)=S(n,r,s)\]
allows us to deduce certain properties of $S(n,r,s)$ from the properties of $S(n,r+(n-1)s)$.
The mistake on page 74 of \cite{dd} was the claim that $\pi''=\pi'$. Since $\pi''$ is a proper subset of $\pi'$ in general, the algebra $S'(n,r,s)$ is not isomorphic to $S(n,r,s)$.

Our main results give presentations of ordinary and rational Schur algebras in terms of generators and relations over fields of arbitrary characteristic.  
The most important tool we are using is the result of Donkin proved in Proposition 1 of subsection 3.3 of \cite{d} stating that the defining ideal $I(\mu)$ of the generalized Schur algebra $S(\mu)$ is generated by its intersection with $Dist(T)$.

We describe $I(\mu)\cap Dist(T)$ in three separate instances. Define 
\[S_0(\mu)=Dist(T)/I(\mu)\cap Dist(T)\]
and  consider a specific subideal $I_0(\mu)$ of $I(\mu)\cap Dist(T)$ given by generators. We show that the quotient $F_0(\mu)=Dist(T)/I_0(\mu)$ is semisimple, and its cardinality does not exceed $|\mu|$. Since there is a surjective morphism $F_0(\mu)\twoheadrightarrow S_0(\mu)$ and the cardinality of $S_0(\mu)$ equals $|\mu|$, we conclude that $F_0(\mu)\simeq S_0(\mu)$ and $I_0(\mu)=I(\mu)\cap Dist(T)$.

The outline of the paper is as follows.
Section 1 introduces the concepts of rational Schur algebras and generalized Schur algebras and describes a bijective correspondence between certain sets of dominant weights.

In Section \ref{2}, we give an explicit presentation of rational Schur algebra $S(n,r,s)$ over fields of characteristic zero and complete the work started in \cite{dd}. 
We also discuss cellular bases for $S(n,r)$ and correct the statement of Theorem 6.5 and Corollary 6.6 of \cite{dd}.

In Sections 3 through 6, we work over a ground field of a positive characteristic.

Section 3 is devoted to preparatory material about binomial expressions.

In Section 4, we determine the distribution algebra of the Frobenius kernels of $G$ by generators and relations.

In Section 5, we determine the generators of the kernel $I(n,r)$ of $\rho_d$
explicitly and obtain the description of $S(n,r)$ as a factoralgebra of $Dist(G)$ by the ideal $I(n,r)$. 

In Section 6, we give a presentation for rational Schur algebras $S(n,r,s)$.

\section{Rational Schur algebras $S(n,r,s)$}

Throughout the paper, $G=\GL(n)$ will denote the general linear group over an algebraically closed ground field $K$
of characteristic zero or $p>0$. 

\subsection{General linear group}
Let $C=\begin{pmatrix}c_{ij}\end{pmatrix}$ be a generic matrix of size $n\times n$, and $D=det(C)$. 
Denote by $A(n)$ the commutative bialgebra generated by the entries $c_{ij}$ of $C$. The comultiplication 
$\Delta:A(n)\to A(n)\otimes A(n)$, given by $c_{ij}\mapsto \sum_{k=1}^n c_{ik}\otimes c_{kj}$ and the counit
$\epsilon:A(n) \to K$, given by $c_{ij}\mapsto \delta_{ij}$ are algebra morphisms. 
The localization of $A(n)$ by the determinant $D$ is the coordinate algebra $K[G]$ of the general linear group $G$. 
It is the functor from commutative algebras to groups such that $G(A)=Hom_K(K[G], A)$.

\subsection{Distribution algebras}

For the definition of the distribution algebra, see \cite{j}. Some papers, in particular, \cite{d}, use hyperalgebra terminology, which can be identified with the distribution algebra. 

Let $T$ be a maximal torus of $G$, corresponding to diagonal matrices. Fix the standard decomposition of the roots $\Phi=\Phi^+\cup\Phi^-$ of $G$ into positive and negative roots, where $\Phi^+=\{\epsilon_i-\epsilon_j|1\leq i<j\leq n\}$. Let $U^+$ and $U^-$ be the corresponding maximal unipotent subgroups of $G$. The triangular decomposition $G=U^+TU^-$ gives the decomposition $Dist(G)=Dist(U^+)Dist(T)Dist(U^-)$.

If the characteristic of $K$ is zero, then $Dist(G)$ is isomorphic to the universal enveloping algebra $U(\mathfrak{g})$ of the general linear Lie algebra $\mathfrak{g}=\mathfrak{gl}_n$.

\subsection{Schur algebras}

The bialgebra $A(n)$ has a natural grading by degrees, and we denote the component of degree $r$ by $A(n,r)$.
The Schur algebra $S(n,r)$ is the finite-dimensional dual $A(n,r)^* = Hom_K(A(n,r),K)$. The modules over the algebra $S(n,r)$ correspond to polynomial modules over $G$ of degree $r$. Every finite-dimensional (rational) $G$-module is obtained from a polynomial module by tensoring with a suitable power of the one-dimensional representation given by the determinant $D$.

Let $E=K^n$ be the natural $G$-module regarded as the set of column vectors. 
As mentioned in the introduction, there is a natural morphism 
\[\rho_d: Dist(G)\to End_K(E^{\otimes d})\]
given by restriction such that its image coincides with the Schur algebra $S(n,d)$.

We denote the image of $Dist(T)$ under $\rho_d$ by $\rho_d(Dist(T))=S_0(n,d)$.

\subsection{Rational Schur algebras}

Irreducible rational $G$-modules correspond to pairs of partitions $(\alpha, \beta)$ such that the sum $l(\alpha)+l(\beta)$ of their lengths does not exceed $n$ - see Proposition 2.1 of \cite{stem}.

Denote the entries of the inverse matrix $C^{-1}$ by $d_{ij}$. It was shown in \cite{dd} that the algebra $K[G]$ coincides with the algebra generated by all $c_{ij}$ and $d_{ij}$.

Denote by $A(n,r,s)$ the subspace of $K[G]$ spanned by all products \[\prod_{i,j} c_{ij}^{a_{ij}} \prod d_{ij}^{b_{ij}},\] where each $a_{ij}, b_{ij}\geq 0$, 
$\sum_{i,j} a_{ij}=r$, and $\sum_{i,j} b_{ij}=s$. Then $A(n,r,s)$ is a subcoalgebra of $K[G]$, and $K[G]=\sum_{r,s} A(n,r,s)$. The last sum is not direct, and we have 
$A(n,r,s)\subset A(n;r+1,s+1)$ for each $r,s$.

The rational Schur algebra $S(n,r,s)$ is the dual algebra \[A(n,r,s)^*=Hom_K(A(n,r,s),K).\]
The rational Schur algebras were defined first in \cite{dd}.
The ordinary Schur algebra $S(n,r)$ is a special case when $s=0$.

Denote by $E_K^{r,s}=E_K^{\otimes r} \otimes {E_K^*}^{\otimes s}$ the mixed tensor space.  
There is a natural morphism 
\[\rho_{r,s}: Dist(G)\to End(E^{\otimes r}\otimes (E^*)^{\otimes s})\]
given by restriction, such that its image coincides with the Schur algebra $S(n,r,s)$ - see $(3.5.4)$ of \cite{dd}.

The map $\rho_{r,s}$ can be realized by assigning to each matrix unit $e_{ij} \in \mathfrak{gl}_n$ the endomorphism $_{ji}D: E_K^{r,s}\to E_K^{r,s}$ induced by the right superderivation $_{ji}D$, which satisfies $(c_{kl})_{ji}D=\delta_{lj}c_{ki}$. 
Additionally, the divided power $e_{ij}^{(k)}$, where $i\neq j$ is assigned to $_{ji}D^{(k)}$,  and the binomial 
$\binom{e_{ii}}{k}$ is assigned to $\binom{_{ii}D}{k}$ - see the end of Section 4 of \cite{lsz}. 

We denote the image of $Dist(T)$ under $\rho_{r,s}$ by $\rho_{r,s}(Dist(T))=S_0(n,r,s)$.

\subsection{Generalized Schur algebras}

We follow the notation of \cite{d}. We write 
$X(n)$ for the set of weights of $G$, $X^+(n)$ for the set of dominant weights in $X(n)$, 
$\Lambda(n,d)=\{\lambda\in X(n)||\lambda|=\sum \lambda_i =d\}$ and $\Lambda^+(n,d)=\Lambda(n,d)\cap X^+(n)$.

Good sources of information about generalized Schur algebras are papers \cite{d2,d3} of Donkin or Appendix A of \cite{j}.
For a subset $\mu$ of $X^+(n)$, we denote by $\mathcal{C}(\mu)$ the category of all $G$-modules with irreducible composition factors $L_G(\lambda)$ for $\lambda\in \mu$. For a $G$-module $M$, we denote by $A(\mu)=O_{\mu}(M)$ the sum of all its submodules that belong to $\mathcal{C}(\mu)$.

We consider $K[G]$ as the $G$-module via the right regular representation $\rho_r$. Then Lemma A.13 of \cite{j} states that $M$ belongs to $\mathcal{C}(\mu)$ if and only if the coefficient space $cf(M)$ is a subset of $O_{\mu}(K[G])$.

The following statement is Proposition A.16 of \cite{j}; it is fundamental for our considerations.
\begin{pr}\label{pr1.1}
Let $\mu\subset X^+(n)$ be finite. Then there is a surjective morphism of associative $K$-algebras
from $Dist(G)$ to $S(\mu)=O_{\mu}(K[G])^*$ given by the restriction. The kernel of this morphism is the two-sided ideal $I(\mu)$ consisting of those elements of $Dist(G)$ that annihilate all $G$-modules in $\mathcal{C}(\mu)$.
\end{pr}

An important description of $I(\mu)$ is given in Proposition 1 of subsection 3.3 of \cite{d}.
\begin{pr}\label{pr1.2}
The ideal $I(\mu)$ is generated by its intersection with $Dist(T)$.
\end{pr}

\subsection{Mixed tensor space and the set of weights $\Lambda(n,r,s)$}

By Lemma 4.2 of \cite{dd}, the set of weights 
of $E_K^{r,s}$ is $\Lambda(n,r,s)$ consisting of $\lambda\in \mathbb{Z}^n$ such that 
\[\sum\{\lambda_i|\lambda_i>0\} = r-t \text{ and } \sum\{\lambda_i|\lambda_i<0\}=t-s \text{ for some }0\leq t\leq \min\{r,s\}.\]

For a weight $\lambda$ denote $P^+_{\lambda}=\{1\leq i\leq n| \lambda_i>0\}$ and 
$P^-_{\lambda}=\{1\leq i\leq n| \lambda_i<0\}$.

\begin{lm}\label{lm1.3}
The set $\Lambda(n,r,s)$ is characterized by the condition $(a)$ and any one of the conditions $(b)$, $(b'),(b''),(b''')$ or $(b'''')$ below.

$(a)$ $\sum_{i=1}^n\lambda_i = r-s$ 

$(b) -s\leq \sum_{i\in S } \lambda_i\leq r$ for every proper subset $S \subset \{1, \ldots, n\}$

$(b')  \sum_{i\in S } \lambda_i\leq r$ for every proper subset $S \subset \{1, \ldots, n\}$

$(b'') -s\leq \sum_{i\in S } \lambda_i$ for every proper subset $S \subset \{1, \ldots, n\}$

$(b''') \sum_{i\in P^+_{\lambda}} \lambda_i \leq r$

$(b'''') -s \leq \sum_{i\in P^-_{\lambda}} \lambda_i$
\end{lm}
\begin{proof}
The fact that $\lambda\in \Lambda(n,r,s)$ is equivalent to conditions $(a)$ and $(b)$ is Lemma 4.4. of \cite{dd}.
The implications $(b)\implies (b')\implies (b''')$ and $(b)\implies (b'')\implies (b'''')$ are trivial.

Assume the conditions $(a)$ and $(b''')$ are satisfied.  Then $\sum_{i\in P^+_{\lambda}} \lambda_i =r-t$ for $t\geq 0$,  
and $-s+t=\sum_{i\in P^-_{\lambda}} \lambda_i$. Since $0\leq r-t$ and $-s+t\leq 0$, we conclude that $\lambda\in \Lambda(n,r,s)$.

Symmetrically, conditions $(a)$ and $(b'''')$ also imply $\lambda\in \Lambda(n,r,s)$.
\end{proof}

According to (4) of \cite{d}, the rational Schur algebra $S(n,r,s)$ coincides with the generalized Schur algebra $S(\pi)$ for 
$\pi=\Lambda^+(n,r,s)$. 

For $\lambda\in \Lambda^+(n,r,s)$, list $P^+_\lambda=\{1, \ldots, u_{\lambda}\}$ 
and $P^-_{\lambda}=\{l_{\lambda}, \ldots, n\}$.

The set $\Lambda^+(n,r,s)$ is characterized in the following lemma.

\begin{lm}\label{lm1.4}
Assume $\mu\in X^+(n)$. Then $\mu\in \Lambda^+(n,r,s)$ if and only if $\mu$ satisfies the condition $(A)$ and one of the conditions $(B), (B'),(B''),(B''')$, and $(B'''')$ below.

$(A) \sum_{i=1}^n \mu_i = r-s$

$(B) -s\leq \sum_{i=1}^k \mu_i \leq r$ for each $1\leq k \leq n$

$(B') \sum_{i=1}^k \mu_i \leq r$ for each $1\leq k <n$

$(B'') -s\leq \sum_{i=1}^k \mu_i $ for each $1< k \leq n$

$(B''') \sum_{i=1}^{u_{\mu}} \mu_i \leq r$

$(B'''') -s \leq \sum_{i=l_{\mu}}^n \mu_i$
\end{lm}
\begin{proof}
The proof is analogous to that of Lemma \ref{lm1.3} and left to the reader.
\end{proof}

\subsection{Connection between $\Lambda^+(n,r,s)$ and $\Lambda^+(n,r+(n-1)s)$}

Denote $R^+_{\lambda}=\{1\leq i\leq n| \lambda_i>s\}$.

Denote 
\[\pi''=\{\lambda\in \Lambda^+(n, r+(n-1)s)|\sum_{i\in R^+_{\lambda}} \lambda_i \leq r+|R^+_{\lambda}|s\}.
\]

Recall that the set 
\[\pi'=\{\lambda=(\lambda_1, \ldots, \lambda_n)\in \Lambda^+(n,r+(n-1)s) | 0\leq \lambda_1, \ldots, \lambda_n\leq r+s\}.\] Then $\pi'' \subseteq \pi'$ but $\pi''$ can be a proper subset of $\pi'$.

Now, we correct the statement from the bottom of p.74 of \cite{dd}.

\begin{lm}\label{lm1.5}
A bijective correspondence exists between weights in $\pi''$ and $\pi=\Lambda^+(n,r,s)$.
\end{lm}
\begin{proof}
Denote $\nu=(1, \ldots, 1)$ and $\mu=\lambda-s\nu$. 

If $\lambda\in \pi''$, then $R^+_{\lambda}=P^+_{\mu}$ and $\sum_{i\in P^+_{\mu}}\mu_i\leq r$. Since 
$\mu_1+\ldots + \mu_n=r-s$, Lemma \ref{lm1.4} implies that $\mu\in \Lambda^+(n,r,s)$.

Conversely, if $\mu\in \Lambda^+(n,r,s)$, then $\lambda=\mu+s\nu\in \Lambda^+(n,r+(n-1)s)$ and $P^+_{\mu}=R^+_{\lambda}$. Additionally, by Lemma \ref{lm1.4}, we have $\sum_{i\in P^+_{\mu}}\mu_i \leq r$, which implies 
$\sum_{i\in R^+_{\lambda}}\lambda_i \leq r+|R^+_{\lambda}|s$, showing $\lambda\in \pi''$. 
\end{proof}

The bijective correspondence between weights in $\pi''$ and $\pi=\Lambda^+(n,r,s)$ is realized for simple modules of $S(\pi'')$ and $S(n,r,s)=S(\Lambda^+(n,r,s))$ by tensoring with the $s$-fold tensor product of the one-dimensional determinantal representation $D$.
Namely, $W$ is the simple module of the highest weight $\lambda$ in $S(\pi'')$ if and only if 
$W\otimes D^{\otimes -s}$ is the simple module of the highest weight $\mu$ in $S(n,r,s)=S(\Lambda^+(n,r,s))$.

To show that $S(\pi')\not\simeq S(\pi)=S(n,r,s)$ in general, Donkin in \cite{d} considers $n=4$, $r=s=1$ and the weight $\lambda=(2,2,0,0)\in \pi'$ in $S(4,4)$.
The corresponding weight $\mu=(1,1,-1,-1)$ does not belong to $S(4,1,1)$ because $\lambda=(2,2,0,0)\notin \pi''$. 
This follows from $|R^+_{\lambda}|=2$ and $\lambda_1+\lambda_2=4>3=r+|R^+_{\lambda}|s$.

\section{Presentation of $S(n,r,s)$ in the characteristic zero case}\label{2}

This section assumes that $char(K)=0$.

It follows from p.101 of \cite{d} that the space $A(\Lambda^+(n,r,s))$ is identical to the coefficient space of $E^{\otimes r} \otimes (E^*)^{\otimes s}$, which is 
$A(n,r,s)$.
It was shown in \cite{d} that the generalized Schur algebra $S(\Lambda^+(n,r,s))$ is isomorphic to the rational Schur algebra $S(n, r,s)$. 
We will switch back and forth between these corresponding objects in what follows.

The main idea behind the presentation of $S(n,r,s)$ is an adaptation and generalization of arguments from the proof of Proposition 4.2 of \cite{dg}.

The algebra $Dist(T)$ is the  polynomial algebra $K[H_1, \ldots, H_n]$ on generators $H_1, \ldots, H_n$ corresponding to $e_{11}, \ldots, e_{nn}$. 

Denote by 
$I_1=I_1(n,r,s)$ the ideal of $Dist(T)$ generated by the elements $\prod_{k=-r}^s (H_i+k)$ for each $i=1, \ldots n$.
Denote by $I_0=I_0(n,r,s)$ the ideal of $Dist(T)$ generated by 
\begin{itemize}
\item $I_1$, 
\item the element $\sum_{i=1}^n H_i-r+s$,
and 
\item
elements $\prod_{k=-r}^s (H_S+k)$ for each proper subset $S$ of $\{1, \ldots, n\}$, where $H_S=\sum_{i\in S} H_i$. 
\end{itemize} 

Further, denote \[F_1=F_1(n,r,s)=Dist(T)/I_1\]  and \[F_0=F_0(n,r,s)=Dist(T)/I_0.\]

Denote by $I(n,r,s)$ the ideal of $Dist(G)$ generated by $I_0(n,r,s)$.
Then $I_0(n,r,s)\subset I(n,r,s)\cap Dist(T)$.

\begin{lm}\label{lm2.1}
The images of $H_i$ of $Dist(T)$ under the map $\rho_{r,s}$ satisfy 
\[\sum_{i=1}^n \rho_{r,s}(H_i)=r-s\] together with the relations 
\[\prod_{k=-r}^s (\rho_{r,s}(H_S)+k)=0\] for each proper subset $S$ of $\{1, \ldots, n\}$, where $H_S=\sum_{i\in S} H_i$.
\end{lm}
\begin{proof}
The set $\Lambda(n,r,s)$ of weights of $E_K^{r,s}$ is characterized by Lemma \ref{lm1.3}. Therefore eigenvalues of all 
diagonalizable operators $\rho_{r,s}(H_S)$ belong to the set $\{-s, \ldots, r\}$, which implies
$\prod_{k=-r}^s (\rho_{r,s}(H_S)+k)=0$. Since every $\lambda\in \Lambda(n,r,s)$ satisfies $\sum_{i=1}^n \lambda_i=r-s$, the relation $\sum_{i=1}^n \rho_{r,s}(H_i)=r-s$ follows.
\end{proof}

The following proposition describes $S_0(n,r,s)$.

\begin{pr}\label{pr2.2}
The algebra $F_0(n,r,s)$ is semisimple and isomorphic to $S_0(n,r,s)$.
\end{pr}
\begin{proof}

By the Chinese remainder theorem, we obtain that \[F_1=\oplus_{-s\leq a_1, \ldots, a_n\leq r} K[H_1]/(H_1-a_1)\otimes \ldots \otimes K[H_n]/(H_n-a_n).\]

This isomorphism is realized by a map that sends a polynomial $p(H_1, \ldots, H_n)$ to $(p(a_1, \ldots, a_n))_{-s\leq a_1, \ldots, a_n\leq r}$.

Denote by $h_{a_1, \ldots, a_n}$ a polynomial in $H_1, \ldots, H_n$ such that $h_{a_1, \ldots, a_n}+I_1$ is the idempotent of $F_1$  corresponding to its direct summand 
\[K[H_1]/(H_1-a_1)\otimes \ldots \otimes K[H_n]/(H_n-a_n).\] 
Then for each $n$-tuple of integers $(b_1, \ldots, b_n)$ such that $-s\leq b_i\leq r$ for each $i=1, \ldots, n$ we have $h_{a_1, \ldots, a_n}(b_1, \ldots, b_n)=\prod_i\delta_{a_i,b_i} \pmod {I_1}$.
The elements $h_{a_1, \ldots, a_n}+I_1$ form a complete set of primitive orthogonal idempotents of $F_1$ that add up to the identity.
Then 
\[F_1=\oplus_{-s\leq a_1, \ldots, a_n\leq r} K(h_{a_1, \ldots, a_n}+I_1)\simeq K^{(r+s+1)^n}\] is a semisimple algebra.
Therefore, $F_0$ is a semisimple algebra as well.

If $\lambda \notin \Lambda(n,r,s)$ then by Lemma \ref{lm1.3} either $\lambda_1+\ldots +\lambda_n\neq r-s$
or there is a proper subset $S$ of $\{1, \ldots, n\}$ such that $\sum_{i\in S} \lambda_i$ is either smaller than $-s$ or
bigger than $r$. In this case, $h_{\lambda_1, \ldots, \lambda_n}\in I_0$.
Therefore, $F_0$ is a subalgebra of $\oplus _{\lambda\in \Lambda(n,r,s)} K(h_{\lambda_1, \ldots, \lambda_n}+I_0)$,
showing that $dim(F_0)\leq |\Lambda(n,r,s)|$, the cardinality of the set $\Lambda(n,r,s)$.

Denote $I=I(n,r,s)$ and $R=Dist(G)/I$. The canonical quotient map $Dist(T)\to R$ induces by restriction to $Dist(T)$  a map
$Dist(T)\to R$. The image of this map is denoted by $R_0$, and its kernel is $I\cap Dist(T)$. Therefore, $R_0 \simeq Dist(T)/(I\cap Dist(T))$.

By Lemma \ref{lm2.1}, there is a surjective map $R_0\to S_0(n,r,s)$ which is a part of the sequence of algebra surjections
\[F_0=Dist(T)/I_0 \to Dist(T)/(I\cap Dist(T)) \to R_0 \to S_0(n,r,s).\]
The first surjection is given by inclusion $I_0\subset I\cap Dist(T)$.

We have $S_0(n,r,s)=(A(n,r,s)\cap T)^*$ and $dim(S_0(n,r,s))$$=dim(A(n,r,s)\cap T)$.
The space $A(n,r,s)\cap T$ has a basis consisting of elements $\prod_{i\in P_{\mu}^+}c_{ii}^{\mu_i}\prod_{i\in P_{\mu}^-} d_{ii}^{-\mu_i}\in E_K^{r,s}$ such that $\sum_{i\in P_{\mu}^+}\mu_i = r-t$, 
$\sum_{i\in P_{\mu}^-}\mu_i=-s+t$ for some $0\leq t\leq \min\{r,s\}$. Therefore
$dim(S_0(n,r,s))=|\Lambda(n,r,s)|\leq dim(F_0)$. Since we have seen earlier that $dim(F_0)\leq |\Lambda(n,r,s)|$, we conclude that the above surjections are algebra isomorphisms.
\end{proof}

The presentation of $S(n,r,s)$ by generators and relations follows from the following statement.

\begin{pr}\label{pr2.3}
The rational Schur algebra $S(n,r,s)$ is a factoralgebra of $Dist(G)$ modulo the ideal $I(n,r,s)$.
\end{pr}
\begin{proof}
By 1.3. (4) of \cite{d}, $S(n,r,s)= S(\Lambda^+(n,r,s))=Dist(G)/(I(\Lambda^+(n,r,s))$.
By Proposition 1.2, the ideal $I(\Lambda^+(n,r,s))$ is generated by the intersection 
\[I(\Lambda^+(n,r,s))\cap Dist(T)\] denoted by $I_T(\Lambda^+(n,r,s)).$
Then $S_0(n,r,s)\simeq Dist(T)/I_T(\Lambda^+(n,r,s))$.
By Proposition \ref{pr2.2}, we obtain that 
\[S_0(n,r,s)\simeq Dist(T)/I_T(\Lambda^+(n,r,s))\simeq Dist(T)/I_0(n,r,s)\simeq F_0\]
is semisimple.

The surjection from $F_0$ to $S_0(n,r,s)$ constructed earlier implies $I_0(n,r,s)\subseteq I_T(\Lambda^+(n,r,s))$.
Therefore, $I_0(n,r,s)=I_T(\Lambda^+(n,r,s))$ and
$I(n,r,s)=I(\Lambda^+(n,r,s))$.
\end{proof}
\begin{cor}\label{cor2.4}
The rational Schur algebra $S(n,r,s)$ is isomorphic to the unital associated algebra given by generators
$e_i, f_i$ for $i=1, \ldots, n-1$, and $H_i$ for $1\leq i\leq n$
subject to the following relations $(a)$ through $(e)$: 

$(a) H_iH_j=H_jH_i$.

$(b) e_if_j-f_je_i=\delta_{ij} (H_j-H_{j+1})$.

$(c) H_ie_j-e_jH_i=(\epsilon_i,\alpha_j)e_j, \quad H_if_j-f_jH_i=-(\epsilon_i,\alpha_j)f_j$.

$(d) H_1+\ldots + H_n=r-s$

$(e) \prod_{k=-r}^s (H_S+k)=0$ for each proper subset $S$ of $\{1, \ldots, n\}$, where $H_S=\sum_{i\in S} H_i$.  
\end{cor}
\begin{proof}
It was proved in Section 6.2 of \cite{dgs} that the defining relations are (a) through (e) together with the Serre relations 
\[e_i^2e_j-2e_ie_je_i+e_je_i^2=0 \text{ if } |i-j|=1; \text{ and } e_ie_j-e_je_i=0 \text{ otherwise},\] \[f_i^2f_j-2f_if_jf_i+f_jf_i^2=0 \text{ if } |i-j|=1; \text{ and } f_if_j-f_jf_i=0 \text{ otherwise. }\]
It was pointed out to me by Stephen Doty that R. Rouquier observed that using Corollary 4.3.2 of \cite{cg}, it is possible to eliminate the above Serre relations, which are one of the defining relations of $U(\mathfrak{gl}_n)$. For details, set $q=1$ in Section 5 of \cite{dg2}.
\end{proof}

\begin{rem}
We can replace condition $(e)$ with a weaker condition as follows:
$(e') \prod_{k=-r}^s (H_S+k)=0$ for each proper subset $S$ of $\{1, \ldots, n\}$ of cardinality smaller or equal to $\frac{n}{2}$.

Denote by $S'=\{1, \ldots, n\}\setminus S$. Using $(d)$, we express $H_{S'}=(r-s)-H_S$.
Then 
\[\prod_{k=-r}^s (H_{S'}+k)=\prod_{k=-r}^s (-H_S+r-s+k)=\prod_{l=-r}^s (-H_S-l)\]
shows that $\prod_{k=-r}^s (H_S+k)=0$ implies $\prod_{k=-r}^s (H_{S'}+k)=0$.

On p.120 of \cite{d},  Donkin proves that Theorem 7.4 of \cite{dd} gives the correct presentation for rational Schur algebra $S(n,r,s)$ for $n=2$ and $n=3$. 
This follows from Corollary \ref{cor2.4} and the above observation.
For $n\geq 4$, Theorem 7.4 of \cite{dd} is incorrect - one needs to add more relations. 
We have discussed Donkin's counterexample at the end of Section 1.
\end{rem}

\section{Cellular bases for $S(n,r,s)$}
In this section, $char(K)$ is arbitrary, and we follow Section 6 of \cite{dd}. It was mentioned in \cite{d} that there is a mistake in the reasoning on the last three lines on page 74 of \cite{dd}. 

We want to point out that this mistake makes Theorem 6.5 and Corollary 6.6 of \cite{dd}, describing the cellular basis of $S(n,r,s)$, also incorrect.  

Donkin also states, "argument for cellularity given there ... relies on the independence of the base field of the dimension of $S(n,r,s)$ and this is justified via the realization of $S(n,r,s)$ as a generalized Schur algebra."
Donkin proved in \cite{d} that $S(n,r,s)$ is a generalized Schur algebra in full generality required.

We now correct the statements of Theorem 6.5 and Corollary 6.6. of \cite{dd}.

\begin{tr}(compare to Theorem 6.5 of \cite{dd})
Consider the Schur algebra $S(n,r+(n-1)s)$ and the set $\nu=\Lambda^+(n,r+(n-1)s)$ of dominant weights ordered by the reverse dominance order $\unrhd$.

Let $\{C^{\lambda}_{S,T}|\lambda\in \nu\}$ be any cellular basis of $S(n,r+(n-1)s)$.
Then the kernel of the quotient map $S(n,r+(n-1)s)\to S(n,r,s)$ is spanned by the set of all $C^{\mu}_{S,T}$ such that 
$\mu=(\mu_1,\ldots, \mu_n)\in \nu$ satisfies $\sum_{i \in R^+_{\mu}} \mu_i > r+|R^+_{\mu}|s$.
\end{tr}
\begin{proof}
We modify the original argument in the proof of Theorem 6.5 of \cite{dd}, using Lemma \ref{lm1.4} and \ref{lm1.5}, and taking into account the fact that $S(n,r,s)=S(\Lambda^+(n,r,s))$ is a generalized Schur algebra.

Denote $A=S(n, r+(n-1)s)$ and $J$  to be the sum of all ideals $A[\unrhd\mu]$ for $\mu\in \nu$ such that 
$\sum_{i \in R^+_{\mu}} \mu_i > r+|R^+_{\mu}|s$. Denote the set of such weights $\mu$ by $\nu'$. 
We will show that $\nu'$ is saturated. Assume $\mu\in \nu'$ and $\kappa$ is a dominant weight such that $\kappa\unrhd \mu$. Then  $\kappa\in \nu$ because $\nu$ is saturated by $(1)$ of \cite{d}.
Additionally, we have
\[\sum_{i=1}^{u_{\mu}} \kappa_i \geq \sum_{i=1}^{u_{\mu}} \mu_i > r+u_{\mu}s.\]
If $u_\kappa\geq u_\mu$, then  
\[\sum_{i \in R^+_{\kappa}} \kappa_i = \sum_{i=1}^{u_\mu} \kappa_i+\sum_{i=u_\mu +1}^{u_{\kappa}} 
\kappa_i>r+u_\mu s+(u_\kappa-u_\mu)s=r+u_\kappa s.\] 
If $u_\kappa<u_\mu$, then $\sum_{i=1}^{u_\mu}\kappa_i>r+u_\mu s$ implies
\[\sum_{i\in R^+_{\kappa}} \kappa_i>r+u_\mu s-\sum_{i=u_{\kappa+1}}^{u_\mu} \kappa_i\geq
r+u_\mu s- (u_\mu-u_{\kappa})s=r+u_{\kappa}s.\]
Together, this shows that $\kappa\in \nu'$.

The ideal $J$, regarded as left $A$-module is filtered by the standard modules $\Delta(\mu)$ for $\mu\in \nu'$.
Therefore, $S_K(n,r,s)$ is a quotient of $A/J$.

The quotient $A/J$ has a filtration by standard modules $\Delta(\lambda)$ for $\lambda\in \nu-\nu'$. Its dimension is
$\sum_{\lambda\in \nu-\nu'} (dim(\Delta(\lambda))^2$.
By Lemma \ref{lm1.5}, there is a bijection between $\nu-\nu'$ and $\Lambda^+(n,r,s)$ that sends 
$\lambda\in \nu-\nu'$ to $\mu=\lambda-s(1, \ldots, 1)\in \Lambda^+(n,r,s)$.
Since $dim(\Delta(\lambda))=dim(\Delta(\mu))$ follows from tensoring with the $s$th power of the determinant, 
we have
$dim(A/J)=\sum_{\mu\in \Lambda^+(n,r,s)} (dim(\Delta(\mu))^2=dim(S(n,r,s))$.
As observed by Donkin, the last equality is true because $S(n,r,s)$ is a generalized Schur algebra, which implies 
that its dimension does not depend on the characteristic of the ground field. 
\end{proof}

\begin{cor}(compare to Corollary 6.6 of \cite{dd})
Under the quotient map 
$S(n,r+(n-1)s) \to S(n,r,s)$, the images of all $C^{\lambda}_{S,T}$, for $\lambda=(\lambda_1, \ldots, \lambda_n)
\in \Lambda^+(n,r+(n-1)s)$ satisfying 
the condition $\sum_{i \in R^+_{\lambda}} \lambda_i \leq r+|R^+_{\lambda}|s$ form a cellular basis of the rational Schur algebra $S_K(n,r,s)$.
\end{cor}

\section{Properties of binomials}\label{3}

We recall some properties of binomial coefficients needed in what follows.

The following statement is taken from \cite{hab}.

\begin{lm}\label{lm3.1}
If $a=a_0+a_1p+\ldots+ a_l p^l $ and $b=b_0+b_1p+\ldots+b_kp^k$ are $p$-adic expansions of $a>0$ and $b>0$
(so that $0\leq a_i,b_j <p$ for each $i$ and $j$), then 
\[\binom{a}{b}\equiv\binom{a_0}{b_0}\binom{a_1}{b_1}\cdots \binom{a_m}{b_m} \pmod p, \text{where } m=\max\{k,l\}.\]
\end{lm}

In a binomial coefficient $\binom{a}{b}$, we always assume that $b>0$ and allow any integral value for $a$.
If $a<0$, then $\binom{a}{b}=(-1)^b \binom{-a+b-1}{b}$, as usual.

The formula for the multiplication of the binomial coefficients $\binom{H}{b}$ and $\binom{H}{a}$ is given in the next statement; see p.91 of \cite{mz0}.

\begin{lm}\label{lm3.2}
Let $a,b>0$. Then 
\[\binom{H}{b}\binom{H}{a}=\binom{H}{a}\binom{H}{b}=\sum_{j=0}^{\min\{a,b\}} \binom{a+b-j}{a-j} \binom{b}{j} \binom{H}{a+b-j}.\]
\end{lm}

We frequently use \emph{Kummer's criterion} stating that for $a>b$,  $\binom{a}{b}$ is a multiple of $p$ if and only if there is a $p$-adic carry when $a-b$ is added to $b$.

Assume that the variable $H$ can assume integral values only.

\begin{lm}\label{lm3.3}
Let $b=b_0+b_1p+\ldots + b_kp^k$ be a $p$-adic expansion of $b>0$. Then 
\[\binom{H}{b}\equiv \binom{H}{b_0}\binom{H}{b_1p}\cdots \binom{H}{b_kp^k} \pmod p.\] 
\end{lm}
\begin{proof}
We use induction on $k$. The statement is trivial for $k=0$. 

Assume it is valid for $k$ and consider
$b=b_0+b_1p+\ldots +b_kp^k+b_{k+1}p^{k+1}$, and $b'=b_0+b_1p+\ldots +b_kp^k$. Then it is enough to verify that \
$\binom{H}{b}\equiv \binom{H}{b'}\binom{H}{b_{k+1}p^{k+1}}\pmod p$.
Using Lemma \ref{lm3.2}, we compute
\[\binom{H}{b'}\binom{H}{b_{k+1}p^{k+1}}=\sum_{j=0}^{b'} \binom{b-j}{b'-j} \binom{b_{k+1}p^{k+1}}{j} \binom{H}{b-j}\equiv \binom{H}{b} \pmod p\]
using Kummer's criterion.
\end{proof}

\begin{lm}\label{lm3.4}
If $0\leq a_i<p$, then $\binom{H}{a_ip^i}$ is a polynomial in $\binom{H}{p^i}$ modulo $p$.
\end{lm}
\begin{proof}
We use induction on $a_i$. The cases when $a_i=0$ or $a_i=1$ are trivial. Assume $a_i>1$, and using Kummer's criterion,
compute
\[\begin{aligned}\binom{H}{(a_i-1)p^i}\binom{H}{p^i}&=\sum_{j=0}^{p^i} \binom{a_ip^i-j}{(a_i-1)p^i-j}\binom{p^i}{p^i-j}\binom{H}{a_ip^i-j}\\
&\equiv \binom{a_i}{a_i-1}\binom{H}{a_ip^i}+\binom{a_i-1}{a_i-2}\binom{H}{(a_i-1)p^i}\pmod p.
\end{aligned}\]
Since $\binom{a_i}{a_i-1}\not \equiv 0 \pmod p$, we express $\binom{H}{a_ip^i}$ as a linear combination 
of $\binom{H}{(a_i-1)p^i}$ and $\binom{H}{(a_i-1)p^i}\binom{H}{p^i}$ modulo $p$. Using the inductive assumption, this is a polynomial in $\binom{H}{p^i}$ modulo $p$.
\end{proof}

\begin{lm}\label{lm3.5}
Assume $l\geq 0$. Then 
$\binom{a}{p^j}\equiv 0\pmod p$ for every $j\geq l$ if and only if $0\leq a<p^l$.
\end{lm}
\begin{proof}
Assume first $a\geq 0$ and consider the $p$-adic expansion of $a=a_0+a_1p+\ldots+a_kp^k$, where $0\leq a_i<p$
for each index $i$. By Lemma \ref{lm3.1}, $\binom{a}{p^j}\equiv 0 \pmod p$ if and only if $a_j=0$. Therefore, we infer that $0\leq a<p^l$.

Assume $a<0$ and $k$ is such that $-a-1<p^k$. Then $\binom{a}{p^k}=\binom{-a-1+p^k}{p^k}\equiv 1 \pmod p$. 
Therefore no $a<0$ can satisfy $\binom{a}{p^j}\equiv 0 \pmod p$ for every $j\geq l$.
\end{proof}
\begin{cor}\label{cor1}
$\binom{a}{p^j}\equiv 0 \pmod p$ for every $j\geq 0$ if and only if $a=0$.
\end{cor}

It is easy to verify the \emph{Pascal's identity} 
\[\binom{a}{b}=\binom{a-1}{b}+\binom{a-1}{b-1}\]
for any $b>1$ and arbitrary $a$.

We require the following two lemmas describing shifts in the binomial expressions. 

\begin{lm}\label{lm3.7}
For $s>0$, we have
\[\binom{H+s}{j}=\sum_{t=0}^j \binom{H}{t}\binom{s}{j-t}.\]
\end{lm}
\begin{proof}
The statement is trivial for $j=0$. Using induction on increasing values of $j$ and $s$, and Pascal identity, we compute
\[\begin{aligned}
\binom{H+s}{j+1}&=\binom{H+s-1}{j+1}+\binom{H+s-1}{j}=\sum_{t=0}^{j+1}\binom{H}{t}\binom{s-1}{j+1-t}
+\sum_{t=0}^j \binom{H}{t}\binom{s-1}{j-t}\\
&=\binom{H}{j+1}+\sum_{t=0}^j \binom{H}{t}\Big[\binom{s-1}{j+1-t}+\binom{s-1}{j-t}\Big]\\
&=\binom{H}{j+1}+\sum_{t=0}^j \binom{H}{t}\binom{s}{j+1-t}
=\sum_{t=0}^{j+1} \binom{H}{t}\binom{s}{j+1-t}.
\end{aligned}\]
\end{proof}

\begin{lm}\label{lm3.8}
For $r>0$, we have
\[\binom{H-r}{j}=\sum_{t=0}^j \binom{H}{t}\binom{-r}{j-t}.\]
\end{lm}
\begin{proof}
For $j=0$, the statement is trivial. Using induction on increasing values of $j$ and $r$, and Pascal identity, we compute
\[\begin{aligned}
\binom{H-r}{j+1}&=\binom{H-r-1}{j+1}-\binom{H-r}{j}=\sum_{t=0}^{j+1}\binom{H}{t}\binom{-r+1}{j+1-t}
-\sum_{t=0}^j \binom{H}{t}\binom{-r}{j-t}\\
&=\binom{H}{j+1}+\sum_{t=0}^j \binom{H}{t}\Big[\binom{-r+1}{j+1-t}-\binom{-r}{j-t}\Big]\\
&=\binom{H}{j+1}+\sum_{t=0}^j \binom{H}{t}\binom{-r}{j+1-t}
=\sum_{t=0}^{j+1} \binom{H}{t}\binom{-r}{j+1-t}.
\end{aligned}\]
\end{proof}

\section{Distribution algebras of Frobenius kernels}
Assume roots $\alpha, \beta\in \Phi$ are such that $\alpha=\epsilon_i-\epsilon_j,\beta=\epsilon_k-\epsilon_l$, where $i\neq j, k\neq l$. Following \cite{dg}, denote $H_{\alpha}=H_i-H_j$, and if $\alpha+\beta\in \Phi$,  then denote  
$c_{\alpha,\beta}=1$ if $j=k$ and $i\neq l$, and $c_{\alpha,\beta}=-1$ if $i=l$, $j\neq k$.

Denote by $G_m$ the $m$th Frobenius kernel of $G=\GL(n)$. The $\mathbb{Z}$-form $_{\mathbb Z}Dist(G_m)$ of the $m$-th Frobenius kernel $G_m$ of $G$ is generated by binomials $\binom{H_i}{k}$ and divided powers 
$x_{\alpha}^{(k)}$ for roots $\alpha=\epsilon_i-\epsilon_j$, where $i\neq j$ and $0\leq k<q=p^m$, see Proposition 4.1 of \cite{hab}.
The (defining and commuting) 
relations (in $_{\mathbb Z}Dist(G_m)$ considered as a $\mathbb{Z}$-module) are given by
\[\begin{aligned}
(a_{i,a,b}): &\binom{H_i}{a}\binom{H_i}{b}=\sum_{j=0}^{\min\{a,b\}} \binom{a+b-j}{a-j}\binom{b}{b-j}\binom{H_i}{a+b-j}\\
&\text{ for } i=1,\ldots, n \text{ and } 0\leq a,b<q\\
(b_{\alpha,k,l}): &x_{\alpha}^{(k)}x_{\alpha}^{(l)}=\binom{k+l}{k} x_{\alpha}^{(k+l)}, x_{\alpha}^{(0)}=1 \text{ for } \alpha\in \Phi \text{ and } 0\leq k,l<q\\
(c_{i,j,a,b}): &\binom{H_i}{a}\binom{H_j}{b}=\binom{H_j}{b}\binom{H_i}{a} \text{ for } 1\leq i,j \leq n  \text{ and }0\leq a,b<q\\
(d_{\alpha,\beta,k,l}): &x_{\alpha}^{(k)}x_{\beta}^{(l)}=\sum_{a=0}^{\min\{k,l\}} c_{\alpha,\beta}^{a} x_{\beta}^{(l-a)}x_{\alpha+\beta}^{(a)}x_{\alpha}^{(k-a)}\\
& \text{ for } 0\leq k,l<q \text{ and } \alpha,\beta\in \Phi \text{ such that } \alpha+\beta\neq 0 \text{ and }
\alpha+\beta\in \Phi,\\
&\text{ where } c_{\alpha, \beta}=1 \text{ for } \alpha=\epsilon_i-\epsilon_j, \beta=\epsilon_j-\epsilon_l, i\neq l\\
&\text{ and } c_{\alpha,\beta}=-1 \text{ for } \alpha=\epsilon_i-\epsilon_j, \beta=\epsilon_k-\epsilon_i, j\neq k \\
(e_{\alpha,k,l}): &x_{\alpha}^{(k)}x_{-\alpha}^{(l)}=\sum_{a=0}^{\min\{k,l\}} x_{-\alpha}^{(l-a)}\binom{H_{\alpha}-k-l+2a}{a}x_{\alpha}^{(k-a)}\\
& \text{  for } 0\leq k,l<q \text{ and } \alpha\in \Phi
\end{aligned}\]
\[\begin{aligned}
(f_{\alpha,\beta,k,l}): &x_{\alpha}^{(k)}x_{\beta}^{(l)}=x_{\beta}^{(l)}x_{\alpha}^{(k)} \text{ for } 0\leq k,l<q \text{ and }\alpha,\beta\in \Phi \text{ such that }\\
&\alpha+\beta\neq 0 \text{ and }\alpha+\beta\notin\Phi
\end{aligned}\] 

If $\alpha\in \Phi$, $\alpha=\epsilon_i-\epsilon_j$ and $0\leq a,k<q$, then
\[\begin{aligned}
(g_{\alpha,k,a}): &\binom{H_i}{a}x_{\alpha}^{(k)}=x_{\alpha}^{(k)}\binom{H_i-k}{a},
\binom{H_j}{a}x_{\alpha}^{(k)}=x_{\alpha}^{(k)}\binom{H_j+k}{a} \\\
&\text{ and } \binom{H_l}{a}x_{\alpha}^{(k)}=x_{\alpha}^{(k)}\binom{H_l}{a}  \text{ for } l\neq i,j\\
(h_{\alpha,k,a}): &x_{\alpha}^{(k)}\binom{H_i}{a}=\binom{H_i+k}{a}x_{\alpha}^{(k)},
x_{\alpha}^{(k)}\binom{H_j}{a}=\binom{H_j-k}{a}x_{\alpha}^{(k)} 
\end{aligned}\]

In the above relations, we apply Lemmas \ref{lm3.7} and \ref{lm3.8} to express $\binom{H_i+k}{a}$ and $\binom{H_j-k}{a}$ for $0\leq a<q$ 
as a $\mathbb{Z}$-linear combinations of $\binom{H_i}{b}$ and $\binom{H_j}{b}$ for $0\leq b<q$.

The relations $(a_{i,a,b})$ follow from the defining relations for binomials
\[(H_i-n)\binom{H_i}{n}=(n+1)\binom{H_i}{n+1}, \binom{H_i}{0}=1\]
by induction. 
The relations $(b_{\alpha,k,l})$ are defining relations of divided powers.
The relation $(c_{i,a,b})$ follow from $H_i H_j=H_j H_i$.
The commutation relations $(d_{\alpha,\beta,k,l}), (e_{\alpha,k,l})$ and $(f_{\alpha,\beta,k,l})$ are listed as 
(5.11b), (5.11a) and (5.11c) in \cite{dg}.
The commutation relations $(g_{\alpha,k,a})$  and $(h_{\alpha,k,a})$ are well-known - see Lemma 26.3 D of \cite{hum}.

Using relations $(a_{i,a,b})$ through $(h_{\alpha,k,a})$, we can write every element of 
$_{\mathbb{Z}}Dist(G_m)$ as a $\mathbb{Z}$-linear combination of terms 
\begin{equation}\tag{*}\prod_{1\leq i<j\leq n} x_{\epsilon_i-\epsilon_j}^{(a_{ij})}\prod_{i=1}^n \binom{H_i}{b_i} \prod_{1\leq j<i\leq n}x_{\epsilon_i-\epsilon_j}^{(c_{ji})},
\end{equation}
where $0\leq a_{ij}, b_i,c_{ij}<q$ and the order of terms in the first and last product is fixed.
We choose the order of the terms in the first product that corresponds to the order of the roots $\alpha=\epsilon_i-\epsilon_j$ given by $\epsilon_i-\epsilon_k<\epsilon_i-\epsilon_l$ if $k<l$ and $\epsilon_i-\epsilon_n<\epsilon_{i+1}-\epsilon_{i+2}$ for 
each $i=1, \ldots, n$.
The order of the terms in the last product is chosen to correspond to the order of the roots $\alpha=\epsilon_i-\epsilon_j$ given by $\epsilon_j-\epsilon_k<\epsilon_j-\epsilon_l$ if $k>l$ and $\epsilon_j-\epsilon_1<\epsilon_{j+1}-\epsilon_{j}$ for 
each $j=1, \ldots, n$.

Since the elements $(*)$ form a $\mathbb{Z}$-basis of $_\mathbb{Z}Dist(G_m)$ by Theorem 1 of \cite{kost}, we conclude that the relations $(a_{i,a,b})$ through $(h_{\alpha,k,a})$ generate all relations amongst the generators of the $\mathbb{Z}$-module $_{\mathbb{Z}}Dist(G_m)$.

After the change of the base field, i.e., tensoring the $\mathbb{Z}$-form ${_{\mathbb{Z}}}Dist(G_m)$ with the field $K$, we obtain that all relations in 
$Dist_{K}(G_m)=K\otimes {_{\mathbb{Z}}}Dist(G_m)$ are generated by relations $(1\otimes a_{i,a,b})$ through $(1\otimes h_{\alpha,k,a})$.

There are finitely many relations $(1\otimes a_{i,a,b})$ through $(1\otimes h_{\alpha,k,a})$. However, some of them are consequences of simpler relations. 

Using Lemmas \ref{lm3.3} and \ref{lm3.4}, we derive that each $\binom{H_i}{a}$ for $0\leq a<q$ is congruent modulo $p$ to a polynomial with integral coefficients in $\binom{H_i}{p^j}$ for $0\leq j<m$.
Therefore, formulas $(1\otimes a_{i,a,b})$ imply that each $1\otimes \binom{H_i}{a}$ for $0\leq a<q$ is a polynomial with integral coefficients in variables $1\otimes \binom{H_i}{p^j}$ for $0\leq j<m$.

The relations $a!x^{(a)}=x^a$ imply that if $a=a_ip^i+a_{i-1}p^{i-1}\ldots +a_0$ is a $p$-adic expansion of $a$, then 
\[x^{(a)}=(x^{(p^i)})^{(a_i)}(x^{(p^{i-1})})^{(a_{i-1})}\ldots x^{(a_0)}=
\frac{(x^{(p^i)})^{a_i}}{a_i!}\frac{(x^{(p^{i-1})})^{a_{i-1}}}{a_{i-1}!}\ldots \frac{x^{a_0}}{a_0!}.
\] Since division by each $(a_j)!$ is well-defined 
over $K$, the commutation relations 
\[(1\otimes c_{i,j,a,b}), (1\otimes d_{\alpha, \beta,k,l}),(1\otimes e_{\alpha,k,l}),(1\otimes f_{\alpha, \beta,k,l}), (1\otimes g_{\alpha,k,a}), (1\otimes h_{\alpha,k,a})\] 
are consequences of $(1\otimes a_{i,a,b})$, $(1\otimes b_{\alpha,k,l})$ and 
\[(1\otimes c_{i,j,p^u,p^v}), (1\otimes d_{\alpha,\beta,p^u,p^v}),(1\otimes e_{\alpha,p^u,p^v}),(1\otimes f_{\alpha, \beta,p^u,p^v}), (1\otimes g_{\alpha,p^u,p^v}), (1\otimes h_{\alpha,p^u,p^v})\] for $0\leq u,v<m$.

We have proved the following lemma.

\begin{lm}\label{lm6.1}
The distribution algebra $Dist_{K}(G_m)$ of the $m$-th Frobenius kernel $G_m$ of $G=\GL(n)$ over the ground field $K$ of characteristic $p>0$ 
is generated by binomials $1\otimes \binom{H_i}{a_i}$ for $i=1, \ldots, n$ and $0\leq a_i<q=p^m$ and divided powers
$1\otimes x_{\alpha}^{(k)}$, where $\alpha\in \Phi$ and $0\leq k<q$.
The relations are generated by
$(1\otimes a_{i,a,b})$, $(1\otimes b_{\alpha,k,l})$, $(1\otimes c_{i,j,p^u,p^u})$, $(1\otimes d_{\alpha,\beta,p^u,p^v})$,$(1\otimes e_{\alpha,p^u,p^v})$, $(1\otimes f_{\alpha, \beta,p^u,p^v})$, $(1\otimes g_{\alpha,p^u,p^v})$ and $(1\otimes h_{\alpha,p^u,p^v})$ for 
$1\leq i,j\leq n$, $0\leq a,b,k,l<q$, $\alpha,\beta\in \Phi$, and $0\leq u,v<m$.
\end{lm}

In what follows, we write $(a_{i,a,b})$ instead of $(1\otimes a_{i,a,b})$ and so on for other relations.

For the integral forms of generalized $q$-Schur algebras, please see Section 8 of \cite{doty}.

\section{Presentation of $S(n,d)$ in positive characteristic case}\label{4}
In this section, we assume the characteristic of the ground field is $p>0$.
Let $m\geq 1$ and $p^m=q$. Denote by $T_m$ the $m$th Frobenius kernel of $T$. 

\begin{pr}\label{pr4.1}
Denote by $J_m$ the ideal of $Dist(T)$ generated by elements $\binom{H_i}{p^j}$ for $i=1, \ldots, n$ and $j\geq m$.
Then $Dist(T)/J_m\simeq Dist(T_m)$ has a basis consisting of elements $\prod_{i=1}^n \binom{H_i}{b_i}$ where  
$0\leq b_i<q$.
\end{pr}
\begin{proof}
The algebra $Dist(T)$ is generated by products $\prod_{i=1}^n \binom{H_i}{a_i}$ for $a_i\geq 0$. By Lemmas \ref{lm3.3} and \ref{lm3.4}, every $\binom{H_i}{a_i}$ is congruent to a polynomial in $\binom{H_i}{b_i}$ for $0\leq b_i<q$, hence an element of $Dist(T_m)$ modulo $J_m$. 
It follows from Lemma \ref{lm3.2} and Kummer's criterion that the set $\prod_{i=1}^n \binom{H_i}{b_i}$ 
for $0\leq b_i<q$ is a $K$-basis of the algebra $Dist(T_m)$. 

We conclude that $Dist(T/J_m)\simeq Dist(T_m)$.
\end{proof}

Following \cite{mz0}, for each $i=1, \ldots, n$ and $0\leq b_i <q$ we define 
\[h_{i, b_i}=\sum_{k_i=b_i}^{q-1} (-1)^{k_i-b_i}\binom{k_i}{b_i}\binom{H_i}{k_i}\]
and 
\[h_{b_1, \ldots, b_n}=\prod_{i=1}^n h_{i,b_i}.\]

\begin{lm}\label{lm4.2}
For $i=1, \ldots, n$ let $0\leq a_i,b_i<q$. Then \[h_{b_1,\ldots, b_n}(a_1, \ldots, a_n)=\prod_{i=1}^n \delta_{a_i,b_i}.\]
Consequently, the elements $h_{b_1, \ldots, b_n}$, where $0\leq b_i <q$ for $i=1, \ldots, n$ form a complete set of primitive orthogonal idempotents that add up to the unity of $Dist(T_m)$, and 
$Dist(T_m)=\oplus_{\stackrel{0\leq b_i<q}{i=1, \ldots n}} Kh_{b_1, \ldots, b_n}$.
\end{lm}
\begin{proof}
It is enough to show that $h_{i,b_i}(a_i)=\delta_{a_i,b_i}$.
If $a_i<b_i$, then $\binom{a_i}{b_i}=0$ and $h_{i,b_i}(a_i)=0$. 
If $a_i=b_i$, then $h_{i,b_i}(a_i)=1$.

Assume $a_i>b_i$. In this case $\binom{k_i}{b_i}\binom{a_i}{k_i}=\binom{a_i}{b_i}\binom{a_i-b_i}{k_i-b_i}$
and 
\[\begin{aligned}h_{i, b_i}(a_i)&=\sum_{k_i=b_i}^{a_i}(-1)^{k_i-b_i}\binom{k_i}{b_i}\binom{a_i}{k_i}
=\sum_{k_i=b_i}^{a_i} (-1)^{k_i-b_i}\binom{a_i}{b_i}\binom{a_i-b_i}{k_i-b_i}\\
&=\binom{a_i}{b_i}\sum_{k_i=b_i}^{a_i}(-1)^{k_i-b_i}\binom{a_i-b_i}{k_i-b_i}=
\binom{a_i}{b_i}\sum_{t=0}^{a_i-b_i}(-1)^{t}\binom{a_i-b_i}{t}=0.
\end{aligned}\]
The second statement follows from $h_{b_1,\ldots, b_n}(a_1, \ldots, a_n)=\prod_{i=1}^n \delta_{a_i,b_i}$.
\end{proof}

Fix an integer $d>0$ and $q=p^m$ such that $d<q$.

Denote by $I_0=I_0(n,d)$ the ideal of $Dist(T)$ generated by $J_m$ and by $\binom{H_1+\ldots + H_n-d}{p^j}$ for all $j\geq 0$.

Denote \[F_1=F_1(n,d)=Dist(T)/J_m\simeq Dist(T_m)\] and \[F_0=F_0(n,d)=Dist(T)/I_0.\]

Further, denote by $I=I(n,d)$ the ideal of $Dist(G)$ generated by $I_0$.

\begin{lm}\label{lm4.3}
Assume $0<d<q$. Then the images of $H_i$ of $Dist(T)$ under the map $\rho_{d}$ satisfy 
\[\binom{\rho_d(H_i)}{p^k}\equiv 0 \pmod p \text{ for all } i=1, \ldots, n \text{ and } k\geq m, and \]
\[\binom{\rho_d(H_1+\ldots + H_n)-d}{p^j}\equiv 0 \pmod p \text{ for all } j\geq 0.\]
\end{lm}
\begin{proof}
The set $\Lambda(n,d)$ of weights of $E_K^d$ is characterized by the conditions 
$0\leq \lambda_i\leq d$ for each $i=1, \ldots, n$ and $\lambda_1+\ldots+\lambda_n=d$.
Since $d<q$, these conditions are equivalent to $0\leq \lambda_i<q$ for each $i=1, \ldots, n$ and $\lambda_1+\ldots+\lambda_n=d$.

Using Corollary \ref{cor1} we infer that $\lambda_1+\ldots+ \lambda_n=d$ implies 
\[\binom{\rho_d(H_1+\ldots + H_n)-d}{p^j}\equiv 0 \pmod p \text{ for all } j\geq 0.\]
By Lemma \ref{lm3.5}, the conditions $0\leq \lambda_i<q$ imply 
\[\binom{\rho_d(H_i)}{p^k}\equiv 0 \pmod p \text{ for all } i=1, \ldots, n \text{ and } k\geq m.\]
\end{proof}

\begin{pr}\label{pr4.4}
Let $0<d<q$. 
Then $S_0(n,d)\simeq F_0(n,d)$.
\end{pr}
\begin{proof}
Recall that $J_m$ is the ideal of $Dist(T)$ generated by elements $\binom{H_i}{p^k}$ for all $i=1, \ldots, n$ and $k\geq m$.

By Lemma \ref{lm4.2}, we have $F_1=\oplus_{\stackrel{0\leq b_i<q}{i=1, \ldots n}} K h_{b_1, \ldots, b_n}=K^{q^n}$ is a semisimple algebra.
Therefore, $F_0$ is a semisimple algebra as well.

If $\lambda$ is such that $0\leq \lambda_i<q$ for $i=1, \ldots, n$ but $\lambda \notin \Lambda(n,d)$, then by 
Corollary \ref{cor1}, we have $\binom{\lambda_1+\ldots +\lambda_n -d}{p^j}\not\equiv 0 \pmod p$ for some $j\geq 0$.
 In this case, $h_{\lambda_1, \ldots, \lambda_n}\in I_0$.
Therefore, $F_0$ is a subalgebra of $\oplus _{\lambda\in \Lambda(n,d)} K(h_{\lambda_1, \ldots, \lambda_n}+I_0)$,
showing that $dim(F_0)\leq |\Lambda(n,d)|$.

Denote $I=I(n,d)$ and $R=Dist(G)/I$. The canonical quotient map $Dist(T)\to R$ induces by restriction to $Dist(T)$  a map
$Dist(T)\to R$. The image of this map is denoted by $R_0$, and its kernel is $I\cap Dist(T)$. Therefore, $R_0 \simeq Dist(T)/(I\cap Dist(T))$.

By Lemma \ref{lm4.3}, there is a surjective map $R_0\to S_0(n,d)$ which is a part of the sequence of algebra surjections
\[F_0=Dist(T)/I_0 \to Dist(T)/(I\cap Dist(T)) \to R_0 \to S_0(n,d).\]
The first surjection is given by inclusion $I_0\subset I\cap Dist(T)$.

We have $S_0(n,d)=(A(n,d)\cap T)^*$ and $dim(S_0(n,d))=dim(A(n,d)\cap T)$.
The space $A(n,d)\cap T$ has a basis consisting of elements $\prod_{i=1}^n c_{ii}^{\lambda_i}\in E_K^{d}$ such that 
$0\leq \lambda_i\leq d$ for each $i=1, \ldots, n$ and $\lambda_1+ \ldots + \lambda_n = d$.
Therefore
$dim(S_0(n,d))=|\Lambda(n,d)|\leq dim(F_0)$. Since $dim(F_0)\leq |\Lambda(n,d)|$, we conclude that the above surjections are algebra isomorphisms.
\end{proof}

The ideal $I_0$ of $Dist(T)$ is generated by infinitely many generators. The following lemma asserts that $I_0$ is not finitely generated. 

\begin{lm}\label{lm4.5}
Assume an ideal $J$ of $Dist(T)$ is finitely generated. Then the factoralgebra $Dist(T)/J$ is infinite-dimensional. 
\end{lm}
\begin{proof}
Each generator of $J$ is a finite linear combination of products of expressions $\binom{H_i}{a}$.
Let $t$ be the maximal number appearing in any such binomial $\binom{H_i}{a}$ and $t<p^k$ for some $k$. Then each generator of $J$ belongs to $Dist(T_k)$, and $J\subset Dist(T_k)$. 
Lemma \ref{lm3.2} together with Kummer's criterion imply that $Dist(T_r)$ is an ideal of $Dist(T)$. 

Since $Dist(T)/Dist(T_k)$ is infinite-dimensional and it is a surjective image of $Dist(T)/J$, we conclude that
$Dist(T)/J$ is also infinite-dimensional.
\end{proof}

\begin{pr}\label{pr4.6}
Assume $d<q=p^m$. Then the Schur algebra $S(n,d)$ is a factoralgebra of $Dist(G)$ modulo the ideal $I(n,d)$.
It is also a factoralgebra of $Dist(G_m)$ of the $m$th Frobenius kernel of $G=\GL(n)$ by the ideal generated by relations $\binom{H_1+\ldots +H_n-d}{p^j}=0$ for
$0\leq j<m+t$, where $t=[\log_p(n)]$.
\end{pr}
\begin{proof}
By 1.3. (4) of \cite{d}, $S(n,d)= S(\Lambda^+(n,d))=Dist(G)/(I(\Lambda^+(n,d))$.
By Proposition 1.2, the ideal $I(\Lambda^+(n,d))$ is generated by the intersection 
\[I(\Lambda^+(n,d))\cap Dist(T)\] denoted by $I_T(\Lambda^+(n,d))$. Then $S_0(n,d)\simeq Dist(T)/I_T(\Lambda^+(n,d))$.
By Proposition \ref{pr4.4}, we have $I_T(\Lambda^+(n,d))=I_0(n,d)$, which implies $I(\Lambda^+(n,d))=I(n,d)$.

Conditions $\binom{H_i}{p^j}=0$ for $j\geq m$ imply $0\leq \lambda_i<q$ for each $i$, which gives  $0\leq \lambda_1+\ldots+\lambda_n<np^m\leq p^{m+t}$. That implies $\binom{H_1+\ldots + H_n-d}{p^j}=0$ for $j\geq m+t$.
\end{proof}

\begin{cor}\label{cor4.7}
Assume $d<p^m=q$. Then the Schur algebra $S(n,d)$ is isomorphic to a factoralgebra of $Dist(G_m)$. 
The unital associated algebra $S(n,d)$ is given by generators $\binom{H_i}{k}$, $x_{\alpha}^{(k)}$ for $\alpha\in \Phi$ 
and $0\leq k<q$. 
The generating relations are
$(a_{i,a,b})$, $(b_{\alpha,a,b})$, $(c_{i,j,p^u,p^v})$, $(d_{\alpha,\beta,p^u,p^u})$, $(e_{\alpha,p^u,p^u})$,
$(f_{\alpha, \beta,p^u,p^v})$, $(g_{\alpha,p^u,p^v})$ and $(h_{\alpha,p^u,p^v})$ for $1\leq i,j\leq n$, $0\leq a,b<q$, $\alpha,\beta\in \Phi$, $0\leq u,v<m$
and additional relations

${(i_j)}: \binom{H_1+\ldots + H_n-d}{p^j}=0$ for all $0\leq j<m+t$, where $t=[\log_p(n)]$.
\end{cor}
\begin{proof}
The proof follows from Lemmas \ref{lm6.1}, \ref{lm4.3} and Proposition \ref{pr4.4}.
\end{proof}

\section{Presentation of $S(n,r,s)$ in positive characteristic case}\label{5}
In this section, we assume the characteristic of the ground field is $p>0$ and $n>1$.

We denote $d=r+(n-1)s$ and choose $m$ so that $d<p^m=q$. Then $r+s<q$ as well.
Recall that $J_m$ is the ideal of $Dist(T)$ generated by elements $\binom{H_i}{p^j}$ for $i=1, \ldots, n$ and $j\geq m$ and 
$I_0(n,d)$ is  the ideal of $Dist(T)$ generated by $J_m$ and by $\binom{H_1+\ldots + H_n-d}{p^j}$ for all $j\geq 0$.

There is an automorphism $\sigma$ of $Dist(T)$ given by $\binom{H_i}{b_i} \mapsto \binom{H_i+s}{b_i}$. Then  
$\sigma(J_m)$ is generated by $\binom{H_i+s}{p^j}$ for $j\geq m$. 

Define $I_\sigma(n,d)$ to be the ideal of $Dist(T)$ generated by $\sigma(J_m)$ together with
$\binom{H_1+\ldots +H_n-r+s}{p^j}$ for all $j\geq 0$.

Further, define $I_0(n,r,s)$ to be the ideal generated by $I_\sigma(n,d)$ together with   
$\binom{H_S+s}{p^j}$ for all $j\geq m$, where $H_S=\sum_{i \in S} H_i$ for each proper subset $S$ of $\{1, \ldots, n\}$. 

The ideal $\sigma^{-1}(I_0(n,r,s))$ is generated by $I_0(n,d)$ and 
$\binom{H_S-|S|s+s}{p^j}$ for all $j\geq m$, where $H_S=\sum_{i\in S} H_i$ 
for each proper subset $S$ of $\{1, \ldots, n\}$.

Then we have the commutative diagram
\[\begin{CD} Dist(T)/J_m@>>>Dist(T)/I_0(n,d)@>>>Dist(T)/\sigma^{-1}(I_0(n,r,s))\\ 
@VV\sigma V @VV\sigma V @VV\sigma V\\ 
Dist(T)/\sigma(J_m) @>>>Dist(T)/I_\sigma(n,d)@>>>Dist(T)/I_0(n,r,s)\end{CD},\]
where vertical maps are isomorphisms induced by $\sigma$ and horizontal maps are surjections.

Finally, denote by $I(n,r,s)$ the ideal of $Dist(G)$ generated by $I_0(n,r,s)$.

\begin{lm}\label{lm5.1}
Assume $r+s<q=p^m$. Then the images of $H_i$ of $Dist(T)$ under the map $\rho_{r,s}$ satisfy 
\[\binom{\rho_{r,s}(H_i)+s}{p^k}\equiv 0 \pmod p \text{ for all } i=1, \ldots, n \text{ and } k\geq m,\]
\[\binom{\rho_{r,s}(H_1+\ldots + H_n)-r+s}{p^j}\equiv 0 \pmod p \text{ for all } j\geq 0, \text{ and}\] 
\[\begin{aligned}&\binom{\rho_{r,s}(H_S)+s}{p^j}\equiv 0 \pmod p \text{ for all }j\geq m, \text{ where } H_S=\sum_{i\in S} H_i\\
&\text{ for each proper subset } S \text{ of }\{1, \ldots, n\}.
\end{aligned}\]
\end{lm}
\begin{proof}
The set $\Lambda(n,r,s)$ of weights $\lambda$ of $E_K^{r,s}$ is characterized by the conditions of Lemma \ref{lm1.3}
stating that $\sum_{i=1}^n \lambda_i=r-s$  and by $-s\leq \sum_{i \in S} \lambda_i\leq r $ for every proper subset $S$ of $\{1, \ldots, n\}$. By Corollary \ref{cor1}, we derive that 
\[\binom{\rho_{r,s}(H_1+\ldots + H_n)-r+s}{p^j}\equiv 0 \pmod p \text{ for all }j\geq 0.\]

Taking one-element sets $S=\{i\}$, we see that $0\leq \lambda_i+s\leq r+s<q$, hence by Lemma \ref{lm3.5}, we derive
\[\binom{\rho_{r,s}(H_i)+s}{p^k}\equiv 0 \pmod p \text{ for all }i=1, \ldots, n \text{ and } k\geq m.\]

The inequalities  $-s\leq \sum_{i \in S} \lambda_i \leq r$ imply $0\leq \sum_{i \in S} \lambda_i +s<q$.
By Lemma \ref{lm3.5}, we get 
\[\binom{\rho_{r,s}(H_S)+s}{p^j}\equiv 0 \pmod p \text{ for all }j\geq m.\]
\end{proof}

Denote $F_0(n,r,s)=Dist(T)/I_0(n,r,s)$.

\begin{pr}\label{pr5.2}
Assume $d<q=p^m$. Then $S_0(n,r,s)\simeq F_0(n,r,s)$.
\end{pr}
\begin{proof}
By Proposition \ref{pr4.1} and Lemma \ref{lm4.2}, we have 
\[ Dist(T_m)\simeq Dist(T)/J_m=\oplus_{\stackrel{0\leq b_i<q}{i=1, \ldots n}} Kh_{b_1, \ldots, b_n}.\]
Therefore, 
\[Dist(T)/\sigma(J_m)=\oplus_{\stackrel{0\leq b_i<q}{i=1, \ldots n}} Kh_{b_1, \ldots, b_n}(H_1+s, \ldots, H_n+s)\]
is a semisimple algebra.

By Lemma \ref{lm1.3}, $\mu\in \Lambda(n,r,s)$ if and only if $\mu_1+\ldots +\mu_n=r-s$ and for each proper subset $S$ of $\{1,\ldots, n\}$ there is $-s\leq \mu_S$ where $\mu_S=\sum_{i\in S} \mu_i$.
If we denote $\lambda=\mu+(s,\ldots, s)$, these conditions are equivalent to 
$(\lambda_1-s)+\ldots + (\lambda_n-s)=r-s$, and for each proper subset $S$ of $\{1, \ldots, n\}$ there is 
$0\leq \lambda_S-|S|s+s<q$ for $\lambda_S=\sum_{i\in S}\lambda_i$.

If $\mu \notin \Lambda(n,r,s)$ then either $\mu_1+\ldots +\mu_n\neq r-s$
or there is a proper subset $S$ of $\{1, \ldots, n\}$ such that $\mu_S<-s$. 
In this case, 
\[h_{\lambda_1,\ldots, \lambda_n}=h_{\mu_1+s,\ldots, \mu_n+s}(H_1+s,\ldots, H_n+s)\in I_0(n,r,s).\]
Therefore, $F_0$ is a subalgebra of 
\[\oplus_{\mu\in \Lambda(n,r,s)} Kh_{\mu_1+s,\ldots, \mu_n+s}(H_1+s, \ldots, H_n+s).\]
In particular, 
$dim(F_0)\leq |\Lambda(n,r,s)|$.

Denote $R=Dist(G)/I(n,r,s)$. The canonical quotient map $Dist(T)\to R$ induces by restriction to $Dist(T)$  a map
$Dist(T)\to R$. The image of this map is denoted by $R_0$, and its kernel is $I(n,r,s)\cap Dist(T)$. Therefore, 
\[R_0 \simeq Dist(T)/(I(n,r,s)\cap Dist(T)).\]

By Lemma \ref{lm5.1}, there is a surjective map $R_0\to S_0(n,r,s)$ which is a part of the sequence of algebra surjections
\[F_0=Dist(T)/I_0(n,r,s) \to Dist(T)/(I(n,r,s)\cap Dist(T))\to R_0\to S_0(n,r,s).\]
The first surjection is given by inclusion $I_0(n,r,s)\subset I(n,r,s)\cap Dist(T)$.

We have $S_0(n,r,s)=(A(n,r,s)\cap T)^*$ and $dim(S_0(n,r,s))=dim(A(n,r,s)\cap T)$.
The space $A(n,r,s)\cap T$ has a basis consisting of elements 
$\prod_{i\in P^+_\lambda} c_{ii}^{\lambda_i}\prod_{i\in P^-_\lambda} d_{ii}^{-\lambda_i}\in E_K^{r,s}$ such that 
$-s\leq \lambda_S\leq r$ for each proper subset $S$ of $\{1, \ldots, n\}$ and $\lambda_1+ \ldots + \lambda_n = r-s$.
Therefore
$dim(S_0(n,r,s))=|\Lambda(n,r,s)|\leq dim(F_0)$. Since $dim(F_0)\leq |\Lambda(n,r,s)|$, we conclude that the above surjections are algebra isomorphisms.
\end{proof}

\begin{rem}\label{rem5.3}
Instead of the above argument, we can work with the sequence 
\[Dist(T)/J_m \to Dist(T)/I_0(n,d) \to Dist(T)/\sigma^{-1}(I_0(n,r,s))\]
of surjections and then apply $\sigma$ to derive that $S_0(n,r,s)\simeq Dist(T)/I_0(n,r,s)$.
This approach is reminiscent of but not the same as in \cite{dd} since the automorphism $\sigma$ of $Dist(T)$ is not given by tensoring with a power of the determinant representation.
\end{rem}

\begin{pr}\label{pr5.4}
Assume $d<q=p^m$. Then the rational Schur algebra $S(n,r,s)$ is a factoralgebra of $Dist(G)$ modulo the ideal $I(n,r,s)$.
\end{pr}
\begin{proof}
By 1.3. (4) of \cite{d}, $S(n,r,s)= S(\Lambda^+(n,r,s))=Dist(G)/(I(\Lambda^+(n,r,s))$.
By Proposition 1.2, the ideal $I(\Lambda^+(n,r,s))$ is generated by the intersection 
\[I(\Lambda^+(n,r,s))\cap Dist(T)\]
denoted by $I_T(\Lambda^+(n,r,s))$.
By Proposition \ref{pr5.2},
\[S_0(n,r,s)\simeq Dist(T)/I_T(\Lambda^+(n,r,s))\simeq Dist(T)/I_0(n,r,s)\simeq F_0(n,r,s)\]
is semisimple.

The surjection from $F_0(n,r,s)$ to $S_0(n,r,s)$ constructed earlier implies $I_0(n,r,s)\subseteq I_T(\Lambda^+(n,r,s))$.
Therefore, $I_0(n,r,s)=I_T(\Lambda^+(n,r,s))$ and
$I(n,r,s)=I(\Lambda^+(n,r,s))$.
\end{proof}

Using the Remark \ref{rem5.3} approach, we can replace $H_i$ with $H_i'=H_i+s$.

The factoralgebra $Dist(T)/\sigma(J_m)\simeq Dist(T_m)$ is given by generators $\binom{H'_i}{a}$ for $0\leq a<q=p^m$ and 
relations 
\[\binom{H'_i}{a}\binom{H'_i}{b}=\sum_{j=0}^{\min\{a,b\}} \binom{a+b-j}{a-j}\binom{b}{b-j}\binom{H'_i}{a+b-j}\]  
for $i=1,\ldots, n$ and $0\leq a,b<q$

We obtain the following presentation of $S(n,r,s)$.

Denote
\[\begin{aligned}
(a'_{i,a,b}): &\binom{H'_i}{a}\binom{H'_i}{b}=\sum_{j=0}^{\min\{a,b\}} \binom{a+b-j}{a-j}\binom{b}{b-j}\binom{H'_i}{a+b-j}\\
& \text{ for } i=1,\ldots, n \text{ and } 0\leq a,b<q\\
(b_{\alpha,k,l}): &x_{\alpha}^{(k)}x_{\alpha}^{(l)}=\binom{k+l}{k} x_{\alpha}^{(k+l)}, x_{\alpha}^{(0)}=1 \text{ for }\alpha\in \Phi \text{ and } 0\leq k,l<q\\
(c'_{i,j,a,b}): &\binom{H'_i}{a}\binom{H'_j}{b}=\binom{H'_j}{b}\binom{H'_i}{a} \text{  for } 1\leq i,j \leq n \text{ and } 0\leq a,b<q\\
(d_{\alpha,\beta,k,l}): &x_{\alpha}^{(k)}x_{\beta}^{(l)}=\sum_{a=0}^{\min\{k,l\}} c_{\alpha,\beta}^{a} x_{\beta}^{(l-a)}x_{\alpha+\beta}^{(a)}x_{\alpha}^{(k-a)}\\
&\text{ for }0\leq k,l<q \text{ and } \alpha,\beta\in \Phi \text{ such that }\alpha+\beta\neq 0 \text{ and }\alpha+\beta\in \Phi\\
(e'_{\alpha,k,l}): &x_{\alpha}^{(k)}x_{-\alpha}^{(l)}=\sum_{a=0}^{\min\{k,l\}} x_{-\alpha}^{(l-a)}\binom{H'_{\alpha}-k-l+2a}{a}x_{\alpha}^{(k-a)}\\
& \text{ for }0\leq k,l<q \text{ and }\alpha\in \Phi, \text{ where } H'_{\alpha}=H'_i-H'_j  \text{ for }\alpha=\epsilon_i-\epsilon_j\\
(f_{\alpha,\beta,k,l}): &x_{\alpha}^{(k)}x_{\beta}^{(l)}=x_{\beta}^{(l)}x_{\alpha}^{(k)} \text{ for } 0\leq k,l<q 
\text{ and }\alpha,\beta\in \Phi  \text{ such that } \alpha+\beta\notin\Phi.
\end{aligned}\] 

If $\alpha\in \Phi$, $\alpha=\epsilon_i-\epsilon_j$ and $0\leq a,k<q$, then
\[\begin{aligned}
(g'_{\alpha,k,a}): &\binom{H'_i}{a}x_{\alpha}^{(k)}=x_{\alpha}^{(k)}\binom{H'_i-k}{a},
\binom{H'_j}{a}x_{\alpha}^{(k)}=x_{\alpha}^{(k)}\binom{H'_j+k}{a}\\
& \text{ and }\binom{H'_l}{a}x_{\alpha}^{(k)}=x_{\alpha}^{(k)}\binom{H'_l}{a} \text{ for }l\neq i,j\\
(h'_{\alpha,k,a}): &x_{\alpha}^{(k)}\binom{H'_i}{a}=\binom{H'_i+k}{a}x_{\alpha}^{(k)},
x_{\alpha}^{(k)}\binom{H'_j}{a}=\binom{H'_j-k}{a}x_{\alpha}^{(k)}.
\end{aligned}\]

(In the above relations, we apply Lemmas \ref{lm3.7} and \ref{lm3.8} to express $\binom{H'_i+k}{a}$ and $\binom{H'_j-k}{a}$ 
as a $\mathbb{Z}$-linear combinations of $\binom{H'_i}{b}$ and $\binom{H'_j}{b}$.)

\begin{cor}\label{cor5.6}
Assume $d<q=p^m$. Then the rational Schur algebra $S(n,r,s)$ is isomorphic to a factoralgebra of $Dist(G)/(\sigma(J_m))$
generated by $\binom{H'_i}{a}$, $x_{\alpha}^{(k)}$ for $\alpha\in \Phi$ and $0\leq a,k<q$ modulo the ideal  generated by
\[\begin{aligned}
&(a'_{l,a,b}),  (b_{\alpha,a,b}), (c'_{l,p^i,p^j}), (d_{\alpha,\beta,p^i,p^j}), (e'_{\alpha,p^i,p^j}), (f_{\alpha, \beta,p^i,p^j}), (g'_{\alpha,p^i,p^j}) \text{ and } (h'_{\alpha,p^i,p^j})\\
&\text{for }1\leq l\leq n, 0\leq a,b<q, 0\leq i,j<m \text{ and }\alpha,\beta\in \Phi \text{ and additional relations}\\
&(i'_j): \binom{H'_1+\ldots + H'_n-d}{p^j}=0 \text{ for all } 0\leq j\leq m+t, \\
& \text{where } t=[\log_p(n)], \text{ and by relations} \\ 
&(j_j): \binom{H'_S-|S|s+s}{p^j}=0  \text{ for every proper subset } S \text{ of }\{1, \ldots, n\}, \\
&\text{where } H'_S=\sum_{i\in S} H'_i, 0\leq j\leq  m+t \text{ and } t=[\log_p(n)].
\end{aligned}\]
\end{cor}
\begin{proof}
Conditions $\binom{H'_i}{p^j}=0$ for $j\geq m$ imply $0\leq \lambda'_i= \lambda_i+s<p^m$ for each $i$, which gives  $-p^m<-d\leq \lambda'_1+\ldots+\lambda'_n-d<np^m-d<np^m\leq p^{m+t}$. Then $\binom{H'_1+\ldots + H'_n-d}{p^{m+t}}=0$ implies $0\leq \lambda'_1+\ldots+\lambda'_n-d$
and $\binom{H'_1+\ldots + H'_n-d}{p^j}=0$ for $j> m+t$.

Also, $-np^m<-ns<H'_S-|S|s+s<np^m$. Then $\binom{H'_S-|S|s+s}{p^{m+t}}=0$ implies $0\leq H'_S-|S|s+s$
and $\binom{H'_S-|S|s+s}{p^j}=0$ for $j> m+t$.
\end{proof}

\section*{Acknowledgement}
The author would like to thank Stephen Doty for valuable suggestions and references.

\end{document}